\documentclass[a4paper,10pt]{article}
\usepackage{amsmath}
\usepackage{amsfonts}
\usepackage{amssymb}
\usepackage{amsthm}

\newtheorem{theorem}{Theorem}

\newcommand{\Qn}{\mathcal{Q}_n}
\newcommand{\subsets}[2]{[#1]^{(#2)}}

\title{A note on balanced independent sets in the cube}
\author{Ben Barber\footnote{Department of Pure Mathematics and Mathematical Statistics, Centre for Mathematical
Sciences, Wilberforce Road, Cambridge, CB3 0WB, UK.  {\tt b.a.barber@dpmms.cam.ac.uk}}}

\begin{document}

\maketitle

\begin{abstract}
Ramras conjectured that the maximum size of an independent set in the discrete cube $\Qn$ containing equal numbers of sets of even and odd size is $2^{n-1}-\binom{n-1}{(n-1)/2}$ when $n$ is odd.  We prove this conjecture, and find the analogous bound when $n$ is even.  The result follows from an isoperimetric inequality in the cube.
\end{abstract}

The discrete hypercube $\Qn$ is the graph with vertices the subsets of $[n]=\{1,\ldots,n\}$ and edges between sets whose symmetric difference contains a single element.  The cube $\Qn$ is bipartite, with classes $X_0$ and $X_1$ consisting of the sets of even and odd size respectively.  The maximum-sized independent sets in $\Qn$ are precisely $X_0$ and $X_1$.  Ramras \cite{Ramras} asked: how large an independent set can we find with half its elements in $X_0$ and half in $X_1$?  Call such an independent set \emph{balanced}.  The following result verifies the conjecture made by Ramras for the case where $n$ is odd.

\begin{theorem}
 The largest balanced independent set in $\Qn$ has size
\begin{align*}
 2^{n-1} - 2 \binom{n-2}{(n-2)/2} & \qquad\text{if $n$ is even,}\\
 2^{n-1} -   \binom{n-1}{(n-1)/2} & \qquad\text{if $n$ is odd.}
\end{align*}
\end{theorem}

For a set $A$ of vertices of $\Qn$, write $N(A)$ for the set of vertices adjacent to an element of $A$.  The maximal independent sets in $\Qn$ all have the form $A\cup (X_1\setminus N(A))$ for some $A\subseteq X_0$.  So for a maximum-sized balanced independent set we seek the largest $A\subseteq X_0$ for which
\[
 |A| \leq |X_1\setminus N(A)|.
\]

We use the following isoperimetric theorem for even-sized sets, due independently to Bezrukov \cite{Bezrukov} and K\"orner and Wei \cite{KornerWei} (see also Tiersma \cite{Tiersma}).  Recall that $x<y$ in the \emph{simplicial} order on $\Qn$ if either $|x| < |y|$, or $|x| = |y|$ and $x<y$ lexicographically.

\begin{theorem}[\cite{Bezrukov}, \cite{KornerWei}]
 Let $A \subseteq X_0$, and let $B$ be the initial segment of the simplicial order restricted to $X_0$ with $|B|=|A|$.  Then $|N(B)| \leq |N(A)|$, and $X_1 \setminus N(B)$ is a terminal segment of the simplicial order restricted to $X_1$.
\end{theorem}

\begin{proof}[Proof of Theorem~1]

We will exhibit an initial segment $A$ of the simplicial order restricted to $X_0$, and a terminal segment $B$ of the simplicial order restricted to $X_1$, with $N(A) \cap B = \emptyset$ and $|A|=|B|$ as large as possible.  It follows from Theorem~2 that $A\cup B$ will be a maximum-sized balanced independent set.

The form of $A$ and $B$ depends on the residue of $n$ mod 4.  For $n=4k$ we take
\begin{align*}
 A & = \subsets{n}{0} \cup \subsets{n}{2} \cup \cdots \cup \subsets{n}{2k-2} \cup (12+\subsets{3,n}{2k-2}) \\
 B & = (1+\subsets{3,n}{2k}) \cup \subsets{2,n}{2k+1} \cup \subsets{n}{2k+3} \cup \cdots \cup \subsets{n}{n-3} \cup \subsets{n}{n-1},
\end{align*}
where, for instance,
\[
 12+\subsets{3,n}{2k-2} = \left\{\{1,2\} \cup x : x \subseteq \{3,4,\ldots,n\}, |x|=2k-2\right\}.
\]
For $n = 4k+1$ we take
\begin{align*}
 A & = \subsets{n}{0} \cup \subsets{n}{2} \cup \cdots \cup \subsets{n}{2k-2} \cup (1+\subsets{2,n}{2k-1}) \\
 B & = \subsets{2,n}{2k+1} \cup \subsets{n}{2k+3} \cup \cdots \cup \subsets{n}{n-2} \cup \subsets{n}{n}.
\end{align*}
For $n=4k+2$ we take
\begin{align*}
 A & = \subsets{n}{0} \cup \subsets{n}{2} \cup \cdots \cup \subsets{n}{2k-2} \cup (1+\subsets{2,n}{2k-1}) \cup (2+\subsets{3,n}{2k-1}) \\
 B & = \subsets{3,n}{2k+1} \cup \subsets{n}{2k+3} \cup \cdots \cup \subsets{n}{n-3} \cup \subsets{n}{n-1}.
\end{align*}
Finally, for $n = 4k+3$ we take
\begin{align*}
 A & = \subsets{n}{0} \cup \subsets{n}{2} \cup \cdots \cup \subsets{n}{2k} \\ 
 B & = \subsets{n}{2k+3} \cup \cdots \cup \subsets{n}{n-2} \cup \subsets{n}{n}.
\end{align*}
Verifying that these sets have the claimed sizes, and that $|A|=|B|$ in each case, is a simple application of the identities $\binom{m}{r} = \binom{m-1}{r-1} + \binom{m-1}{r}$, $\binom{m}{r} = \binom{m}{m-r}$ and $\sum_{r=0}^{m} \binom{m}{r} = 2^m$.
\end{proof}

The maximum-sized balanced independent sets constructed above are also maximal independent sets.  For example, if $n=4k+3$, then any set not in the family is adjacent to a complete layer; the other cases are similar, with slight complications in the middle layers of the cube.

\end{document}